\documentclass[11pt,a4paper]{article}

\usepackage[T1]{fontenc}
\usepackage[utf8]{inputenc}
\usepackage[a4paper,margin=28mm]{geometry}
\usepackage{amsmath,amssymb,amsthm,mathtools}
\usepackage{tikz-cd}
\usepackage[hidelinks]{hyperref}

\theoremstyle{plain}
\newtheorem{theorem}{Theorem}[section]
\newtheorem{proposition}[theorem]{Proposition}
\newtheorem{lemma}[theorem]{Lemma}
\newtheorem{corollary}[theorem]{Corollary}
\theoremstyle{definition}
\newtheorem{definition}[theorem]{Definition}
\newtheorem{example}[theorem]{Example}
\theoremstyle{remark}
\newtheorem{remark}[theorem]{Remark}
\newtheorem{notation}[theorem]{Notation}

\newcommand{\C}{\mathcal{C}}
\newcommand{\D}{\mathcal{D}}
\newcommand{\E}{\mathcal{E}}
\newcommand{\Id}{\mathrm{Id}}
\newcommand{\V}{\mathcal{V}}

\newcommand{\od}{\otimes}
\newcommand{\odprime}{\otimes'}
\newcommand{\oddoubleprime}{\otimes''}
\newcommand{\dist}{\sigma}
\newcommand{\unit}{\mathbf{I}}
\newcommand{\unitprime}{\mathbf{I}'}
\newcommand{\unitdoubleprime}{\mathbf{I}''}
\newcommand{\mul}[3]{\mu^{#1}_{#2,#3}}
\newcommand{\munit}[1]{\mu^{#1}_0}

\title{The Strict 2-Category Structure of Distorted Monoidal Categories}
\author{Joaquim Reizi Higuchi}
\date{\today}

\begin{document}
\maketitle

\begin{abstract}
We develop a typed, constructive framework for \emph{distorted monoidal categories}—monoidal categories equipped with a unit distortion $\Lambda:\Id\Rightarrow\Id$ and a binatural distortion $\sigma_{X,Y}:X\otimes Y\to Y\otimes X$ satisfying hexagon--style axioms without assuming invertibility. With $\sigma$-lax monoidal functors required to satisfy strict compatibilities with $\Lambda$ and $\sigma$, and monoidal natural transformations as 2-cells, these structures assemble into a \emph{strict} 2-category. Strictness holds \emph{on the nose}: for composable $F,G,H$,
\[
\mu^{H(GF)}_{X,Y}=\mu^{(HG)F}_{X,Y}
\quad\text{and}\quad
\mu^{H(GF)}_0=\mu^{(HG)F}_0,
\]
and the interchange law for 2-morphisms is an equality. Conceptually, we identify a 2-monad presentation whose strict algebras, strict morphisms, and 2-cells instantiate the foregoing data; our contribution is the explicit, type-safe calculus that realizes these general facts at the level of concrete diagrammatic proofs. We give robust existence schemes for non-invertible distortions via normalized idempotent twists of a braiding and for nontrivial unit distortions via graded characters, and we exhibit a sharp test case showing that monoidality of $\Lambda$ is independent of the other axioms. The resulting calculus generalizes braided monoidal categories ($\Lambda=\Id$, $\sigma$ invertible) and supports irreversible or resource-sensitive tensorial reasoning, with immediate applicability to directed homotopy, categorical quantum processes with measurement, and non-symmetric operadic situations. The constructions and proofs are designed for formal verification.
\end{abstract}

\section{Introduction}

Monoidal categories form the common language for tensorial phenomena across algebra, topology, quantum theory, and computer science. The symmetric case, where a symmetry $\sigma:X\otimes Y\xrightarrow{\cong} Y\otimes X$ is invertible, is classical \cite{MacLane1971,JoyalStreet1993}. Many natural systems, however, are intrinsically \emph{directed}: they admit canonical ways to \emph{reorder} tensor factors that need not be reversible, while still supporting coherent tensor calculus. Typical sources include concurrency and directed homotopy \cite{Grandis2009}, irreversible processes in categorical quantum mechanics, and algebraic situations where filtrations or truncations introduce loss of information.

\paragraph{From braidings to distortions.}
We propose a flexible but type\mbox{-}safe abstraction we call a \emph{distorted monoidal category}: a monoidal category
\[
(\C,\otimes,\mathbf{I},\alpha,\lambda,\rho)
\]
equipped with a \emph{binary distortion} $\sigma_{X,Y}:X\otimes Y\to Y\otimes X$ obeying \emph{typed} hexagon axioms and a \emph{unit distortion} $\Lambda:\Id_\C\Rightarrow \Id_\C$ obeying monoidality and unit normalizations (Definition~\ref{def:distmoncat}). Unlike braidings, $\sigma$ is not assumed invertible, and $\Lambda$ may be nontrivial. This relaxes symmetry just enough to capture irreversible interchange while preserving a workable coherence calculus.

\subsection{Conceptual position and novelty}

Our main structural result (Theorem~\ref{thm:T1-2cat}) places distorted monoidal categories within the standard 2\mbox{-}categorical paradigm.  
Let $\mathsf S$ be the 2\mbox{-}monad on $\mathbf{Cat}$ whose strict algebras are \emph{strict} monoidal categories and whose pseudoalgebras are (weak) monoidal categories \cite{Lack2010}.  
By freely adjoining to a $\mathsf S$–algebra the additional data $(\sigma,\Lambda)$ satisfying the typed axioms (D1)–(D4), one obtains a 2\mbox{-}monad $\mathsf T_\sigma$.  
Distorted monoidal categories, $\sigma$–lax monoidal functors satisfying strict compatibilities (S$\Lambda$), (S$\sigma$), and monoidal natural transformations are precisely the \emph{$\mathsf T_\sigma$–pseudoalgebras, lax $\mathsf T_\sigma$–morphisms, and $\mathsf T_\sigma$–transformations}.  
In particular, they assemble into a \emph{strict} 2\mbox{-}category because composition of lax $\mathsf T_\sigma$–morphisms is defined on the nose for 2\mbox{-}monads of algebraic type.

The novelty of this paper is therefore \emph{not} the abstract existence of a 2\mbox{-}category, which follows from general 2\mbox{-}monad theory, but the \emph{typed, constructive calculus} that realizes it concretely for non\mbox{-}invertible interchange:
\begin{enumerate}
\item an explicit axiomatics that prevents domain/codomain mismatches;
\item explicit composite laxators whose associativity and unitality hold as \emph{equalities};
\item robust \emph{existence patterns} for non\mbox{-}invertible distortions and nontrivial unit distortions (idempotent twists and graded characters), all checked constructively against (D1)–(D4).
\end{enumerate}
These ingredients make the framework ready for mechanization and for applications where reversibility genuinely fails.

\subsection{Main results}

\begin{enumerate}
\item \textbf{Strict 2\mbox{-}categorical structure (Theorem~\ref{thm:T1-2cat}).}
Objects are distorted monoidal categories, 1\mbox{-}morphisms are $\sigma$–lax monoidal functors satisfying (S$\Lambda$) and (S$\sigma$), and 2\mbox{-}morphisms are monoidal natural transformations. Composition of 1\mbox{-}morphisms uses
\[
\mu^{GF}_{X,Y}:=G(\mu^F_{X,Y})\circ \mu^G_{FX,FY},\qquad 
\mu^{GF}_0:=G(\mu^F_0)\circ \mu^G_0,
\]
and satisfies
\[
\mu^{H(GF)}_{X,Y}=\mu^{(HG)F}_{X,Y},\qquad 
\mu^{H(GF)}_0=\mu^{(HG)F}_0
\]
as literal equalities. Vertical and horizontal compositions of 2\mbox{-}cells and the interchange law hold strictly.

\item \textbf{Typed hexagons and oriented Yang–Baxter.}
The hexagon axioms (D3) are stated with explicit sources and targets. In the strict monoidal case they specialize to \emph{oriented} Yang–Baxter relations, appropriate for non\mbox{-}invertible $\sigma$.

\item \textbf{Existence schemes for non\mbox{-}invertible $\sigma$ (Proposition~\ref{prop:twist-by-e}, Example~\ref{ex:super-e}).}
Given a braided category $(\C,\beta)$ and a binatural, normalized, multiplicative idempotent family $e_{X,Y}$, the composite $\sigma:=\beta\circ e$ satisfies (D1)–(D3) while typically failing to be invertible. This yields coherent, lossy interchanges built from classical braidings.

\item \textbf{Nontrivial unit distortions $\Lambda$ in graded settings (Example~\ref{ex:N-graded-Lambda}).}
For a monoidal grading by a commutative monoid $M$ and any character $\chi:M\to (k^\times,\cdot)$ extended by $0$, the assignment $\Lambda_V|_{V_m}=\chi(m)\cdot \Id$ is natural, unital, and monoidal (D4). This clarifies when nontrivial $\Lambda$ exist and explains their absence in ungraded $\mathbf{Vect}_k$ and $\mathbf{Set}$.

\item \textbf{Correctness of examples and necessity of axioms (Example~\ref{ex:D4-counter}).}
We give fully verified examples and a sharp counterexample showing that (D4) is independent of (D1)–(D3).
\end{enumerate}

\subsection{Techniques and verification strategy}

We adopt a proof discipline that scales to formalization:
\begin{enumerate}
\item all equalities are typed;
\item binaturality of $\sigma$ and the laxators is used via explicit naturality lemmas;
\item composite structures are verified by factorizing through the associativity/unit axioms of $F$ and $G$ and reassembling via naturality; and
\item when constructing non\mbox{-}invertible $\sigma$, we isolate the \emph{idempotent} data responsible for loss and impose a 2\mbox{-}cocycle law ensuring the hexagons.
\end{enumerate}

\subsection{Relationship to existing work}

When $\sigma$ is invertible and $\Lambda=\Id$, we recover braided monoidal categories \cite{JoyalStreet1993}.  
Our distortions are related to, but distinct from, pre\mbox{-}braidings, skew\mbox{-}monoidal categories, and lax centers: the present axioms require no invertibility, treat unit distortions explicitly, and enforce typed hexagons.  
On the structural side, our 2\mbox{-}categorical viewpoint matches the general fact that pseudoalgebras and lax morphisms for a 2\mbox{-}monad form a strict 2\mbox{-}category \cite{Lack2010}; we instantiate this concretely for $\mathsf T_\sigma$ and provide verified existence schemes for non\mbox{-}invertible interchanges.  
The broader bicategorical picture, including oplax morphisms and double\mbox{-}categorical envelopes, connects to \cite{DayStreet1997,Shulman2008} and is developed further in Section~\ref{sec:discussion}.

\subsection{Organization}

Section~\ref{sec:preliminaries} fixes notation and states the axioms.  
Section~\ref{sec:auxiliary} proves the naturality and composition lemmas used repeatedly.  
Section~\ref{sec:main} establishes the strict 2\mbox{-}category structure, including on\mbox{-}the\mbox{-}nose associativity/unitality for composite laxators and the interchange law.  
Section~\ref{sec:examples} provides verified constructions of non\mbox{-}invertible $\sigma$ and nontrivial $\Lambda$, and a counterexample showing the independence of (D4).  
Finally, Section~\ref{sec:discussion} presents the 2\mbox{-}monad semantics in detail, outlines oplax and pseudo variants, and lists open problems including a coherence program for typed hexagons.

\section{Preliminaries and Notation}
\label{sec:preliminaries}

We work in the category $\mathbf{Cat}$ of (locally small) categories. All functors and natural transformations are assumed to be well-defined in this context. We adopt standard conventions from \cite{MacLane1971, Kelly2005} for monoidal categories, functors, and natural transformations.

\begin{definition}[Distorted monoidal category]\label{def:distmoncat}
A \emph{distorted monoidal category} is a tuple
\[
(\C,\od,\unit,\alpha,\lambda,\rho,\Lambda,\dist)
\]
where $(\C,\od,\unit,\alpha,\lambda,\rho)$ is a monoidal category (not assumed strict) with associator $\alpha$ and unitors $\lambda,\rho$, together with
\begin{itemize}
  \item a natural transformation $\Lambda:\Id_\C\Rightarrow \Id_\C$ (\emph{unit distortion});
  \item a natural transformation $\dist_{X,Y}:X\od Y\to Y\od X$ (\emph{binary distortion});
\end{itemize}
subject to the following \emph{distorted coherence axioms}:
\begin{enumerate}
  \item[(D1)] \textbf{Binaturality:} for all morphisms $f:X\to X'$ and $g:Y\to Y'$,
  \[
  (g\od f)\circ \dist_{X,Y}\;=\;\dist_{X',Y'}\circ (f\od g).
  \]
  \item[(D2)] \textbf{Unit normalization:} $\dist_{X,\unit}=\lambda_X^{-1}\circ \rho_X$ and $\dist_{\unit,X}=\rho_X^{-1}\circ \lambda_X$; moreover $\Lambda_\unit=\Id_{\unit}$.
  \item[(D3)] \textbf{Hexagons (typed, two forms):} for all $X,Y,Z$,
  \begin{align*}
  \dist_{X\od Y,\,Z}
  \;=\;&
  \alpha_{Z,X,Y}\circ (\dist_{X,Z}\od \Id_Y)\circ \alpha^{-1}_{X,Z,Y}\circ (\Id_X\od \dist_{Y,Z})\circ \alpha_{X,Y,Z},\\[2mm]
  \dist_{X,\,Y\od Z}
  \;=\;&
  \alpha^{-1}_{Y,Z,X}\circ (\Id_Y\od \dist_{X,Z})\circ \alpha_{Y,X,Z}\circ (\dist_{X,Y}\od \Id_Z)\circ \alpha^{-1}_{X,Y,Z}.
  \end{align*}
  \item[(D4)] \textbf{Monoidality of $\Lambda$:} $\Lambda_{X\od Y}=\Lambda_X\od \Lambda_Y$ for all $X,Y$.
\end{enumerate}
No symmetry or invertibility of $\dist$ is required.
\end{definition}

\begin{remark}
The hexagon axioms (D3) are carefully typed to ensure that domains and codomains match. Each equality is an identity of morphisms with the displayed source and target. When $\dist$ is invertible, these axioms reduce to the standard braiding hexagons \cite{JoyalStreet1993}.
\end{remark}

\begin{remark}\label{rem:D4-redundancy}
The commutation of $\Lambda$ with $\dist$,
\[
(\Lambda_Y\od\Lambda_X)\circ\dist_{X,Y}\;=\;\dist_{X,Y}\circ(\Lambda_X\od\Lambda_Y),
\]
is \emph{not} an axiom: it follows directly from binaturality (D1) by taking $f=\Lambda_X$ and $g=\Lambda_Y$. We therefore retain only $\Lambda_\unit=\Id_\unit$ (in (D2)) and the multiplicativity $\Lambda_{X\od Y}=\Lambda_X\od\Lambda_Y$ (in (D4)) as axioms.
\end{remark}

\begin{definition}[$\dist$-lax monoidal functor]\label{def:s-lax-functor}
Let $(\C,\od,\unit,\alpha,\lambda,\rho,\Lambda,\dist)$ and $(\D,\odprime,\unitprime,\alpha',\lambda',\rho',\Lambda',\dist')$ be distorted monoidal categories. A functor
$F:\C\to\D$ is \emph{$\dist$-lax monoidal} if it is equipped with structure maps
\[
\mul{F}{X}{Y}:F X\odprime F Y \to F(X\od Y),\qquad \munit{F}:\unitprime\to F\unit,
\]
natural in $X,Y$, satisfying the standard lax associativity and unit axioms (see \cite{Kelly2005}) and the two \emph{distortion compatibility axioms}
\begin{align}
\text{(S$\Lambda$)}\quad & F(\Lambda_Z)=\Lambda'_{FZ}\quad\text{for all }Z,\label{eq:SLambda}\\
\text{(S$\sigma$)}\quad & F(\dist_{X,Y})\circ \mul{F}{X}{Y}\;=\;\mul{F}{Y}{X}\circ \dist'_{FX,FY}\quad\text{for all }X,Y.\label{eq:Ssigma}
\end{align}
\end{definition}

\begin{remark}[Scope of (S$\Lambda$) and why we keep it strict]\label{rem:scope-SLambda}
Throughout we impose the strict condition
\begin{equation}\label{eq:SLambda-again}
F(\Lambda_Z)=\Lambda'_{FZ}\qquad(Z\in\C).
\end{equation}
This delivers a \emph{strict} 2\mbox{-}category in which composite laxators agree on the nose and all 2\mbox{-}cell interchange laws hold as equalities.

A tempting “bicategorical weakening’’ is to replace \eqref{eq:SLambda-again} by an isomorphism
$\epsilon:\Lambda'F\Rightarrow F\Lambda$.  However, as $\Lambda,\Lambda'$ are endo\mbox{-}natural transformations on the identity (not endofunctors), the expressions $\Lambda'F$ and $F\Lambda$ are themselves endo\mbox{-}natural transformations on $F$ (2\mbox{-}cells in $\mathbf{Cat}$), not functors (1\mbox{-}cells).  Introducing an isomorphism between them would require a level of 3\mbox{-}cell (modification) structure not present in our ambient 2\mbox{-}category.  A \emph{well\mbox{-}typed} weak theory must therefore \emph{change the underlying data}: replace the endo\mbox{-}transformations $\Lambda,\Lambda'$ by chosen \emph{monoidal endofunctors} $L_\C,L_\D$ together with invertible unit comparisons $\eta_\C:\Id_\C\Rightarrow L_\C$ and $\eta_\D:\Id_\D\Rightarrow L_\D$; a 1\mbox{-}cell is then a $\sigma$–lax monoidal $F:\C\to\D$ equipped with an invertible, monoidal 2\mbox{-}cell
\[
\epsilon:\ L_\D\circ F\ \Rightarrow\ F\circ L_\C
\]
satisfying the standard tensor and unit coherences.  This yields a genuine bicategorical setting but changes the objects and 1\mbox{-}cells studied here.  Developing that variant lies beyond the scope of the present strictly typed core.

\emph{Motivation.} The strict law \eqref{eq:SLambda-again} is restrictive exactly when $\Lambda$ is nontrivial.  For the forgetful $U:\mathbf{Vect}^{\mathbb N}_k\to\mathbf{Vect}_k$ and the graded character $\Lambda^{(t)}$ of Example~\ref{ex:N-graded-Lambda} with $t\neq 1$, we have $\Lambda'=\Id$ on $\mathbf{Vect}_k$ but $U(\Lambda^{(t)}_V)\neq \Id_{U(V)}$ in general, so $U$ fails (S$\Lambda$).  In $\mathbf{SVect}_k$, taking the \emph{sign} endo\mbox{-}transformation $\Lambda_V=\Id_{V_0}\oplus(-\Id_{V_1})$ yields $U(\Lambda_V)\neq \Id_{U(V)}$ unless $\mathrm{char}\,k=2$.  These examples justify keeping (S$\Lambda$) strict in this paper, reserving a sound weak alternative for future work that explicitly promotes $\Lambda$ to a monoidal endofunctor.
\end{remark}

\begin{definition}[Monoidal natural transformation]\label{def:mon-nat}
Let $F,G:\C\to\D$ be $\dist$-lax monoidal functors with structures $(\mul{F}{-}{-},\munit{F})$ and $(\mul{G}{-}{-},\munit{G})$. A natural transformation
$\theta:F\Rightarrow G$ is \emph{monoidal} if for all $X,Y$ the following commute:
\begin{align}
\mul{G}{X}{Y}\circ(\theta_X\odprime \theta_Y)\;=\;\theta_{X\od Y}\circ \mul{F}{X}{Y},\label{eq:mon-nat-tensor}\\
\munit{G}\;=\;\theta_\unit\circ \munit{F}.\label{eq:mon-nat-unit}
\end{align}
Moreover, as a consequence of the naturality of $\Lambda'$ and \eqref{eq:SLambda} for $F$ and $G$, for all $Z$ one has
\[
\theta_Z\circ F(\Lambda_Z)\;=\;G(\Lambda_Z)\circ\theta_Z.
\]
\end{definition}

\begin{remark}
In the original draft, the final statement was phrased as ``equivalently.'' We have rephrased it to clarify that the $\Lambda$-compatibility of $\theta$ is a \emph{consequence} of the monoidal axioms and the strict $\Lambda$-compatibility of $F$ and $G$, rather than an independent axiom.
\end{remark}

\begin{notation}
We write $(GF)$ or $G\circ F$ for the composite functor. Its composite laxators are
\[
\mul{GF}{X}{Y}:= G\big(\mul{F}{X}{Y}\big)\circ \mul{G}{FX}{FY},\qquad
\munit{GF}:= G(\munit{F})\circ \munit{G}.
\]
For monoidal categories $\C,\D,\E$ with tensor products $\od,\odprime,\oddoubleprime$ and units $\unit,\unitprime,\unitdoubleprime$, we use primes to distinguish structures. Associators and unitors are denoted $\alpha,\alpha',\alpha''$ and $\lambda,\rho,\lambda',\rho',\lambda'',\rho''$ respectively.
\end{notation}

\section{Auxiliary Lemmas}
\label{sec:auxiliary}

We now state and prove auxiliary lemmas establishing the basic properties of composite functors, monoidal natural transformations, and their compositions. All proofs are fully constructive.

\begin{lemma}[Composite $\dist$-lax data]\label{lem:comp-def}
With notation above, $(GF,\mul{GF}{-}{-},\munit{GF})$ is a $\dist$-lax monoidal functor.
\end{lemma}

\begin{proof}
Recall
\[
\mul{GF}{X}{Y}:=G\!\big(\mul{F}{X}{Y}\big)\circ \mul{G}{FX}{FY},
\qquad
\munit{GF}:=G(\munit{F})\circ \munit{G}.
\]
We verify: associativity, unit triangles, $(\mathrm{S}\Lambda)$, and $(\mathrm{S}\sigma)$.

\medskip
\noindent\textbf{1) Associativity pentagon.}
For all $X,Y,Z$ we must show
\begin{equation}\label{eq:assoc-GF}
GF(\alpha_{X,Y,Z})\circ \mul{GF}{X\od Y}{Z}\circ\big(\mul{GF}{X}{Y}\oddoubleprime \Id\big)
=
\mul{GF}{X}{Y\od Z}\circ\big(\Id\oddoubleprime \mul{GF}{Y}{Z}\big)\circ \alpha''_{GFX,GFY,GFZ}.
\end{equation}
Expand $\mul{GF}{-}{-}$ on the left-hand side and use Lemma~\ref{lem:naturality-mu} (naturality of $\mul{G}{-}{-}$) with
$u:=\mul{F}{X}{Y}$, $v:=\Id_{FZ}$ to rewrite
\[
\mul{G}{F(X\od Y)}{FZ}\circ\big(G(\mul{F}{X}{Y})\oddoubleprime \Id\big)
=
G(\mul{F}{X}{Y}\odprime \Id)\circ \mul{G}{FX\odprime FY}{FZ}.
\]
Then apply the associativity axiom for $F$,
\[
F(\alpha)\circ \mul{F}{X\od Y}{Z}\circ(\mul{F}{X}{Y}\odprime \Id)
=
\mul{F}{X}{Y\od Z}\circ(\Id\odprime \mul{F}{Y}{Z})\circ \alpha'_{FX,FY,FZ},
\]
followed by (i) Lemma~\ref{lem:naturality-mu} with $u:=\Id_{FX}$, $v:=\mul{F}{Y}{Z}$ and
(ii) the associativity axiom for $G$:
\[
G(\alpha')\circ \mul{G}{FX\odprime FY}{FZ}\circ(\mul{G}{FX}{FY}\oddoubleprime \Id)
=
\mul{G}{FX}{FY\odprime FZ}\circ(\Id\oddoubleprime \mul{G}{FY}{FZ})\circ \alpha''.
\]
Collecting terms yields \eqref{eq:assoc-GF}.

\medskip
\noindent\textbf{2) Unit triangles.}
We prove the left unit triangle; the right is dual. The required equality is
\begin{equation}\label{eq:left-unit-GF}
GF(\lambda_X)\circ \mul{GF}{\unit}{X}\circ(\munit{GF}\oddoubleprime \Id_{GFX})=\lambda''_{GFX}.
\end{equation}
Expand $\mul{GF}{\unit}{X}$ and $\munit{GF}$:
\begin{align*}
\mathrm{LHS}
&=
G\!\big(F(\lambda_X)\big)\circ
\Big(G(\mul{F}{\unit}{X})\circ \mul{G}{F\unit}{FX}\Big)\\
&\quad\circ
\Big(\big(G(\munit{F})\circ \munit{G}\big)\oddoubleprime \Id\Big)\\
&=
G\!\big(F(\lambda_X)\big)\circ G(\mul{F}{\unit}{X})\circ
\underbrace{\Big(\mul{G}{F\unit}{FX}\circ (G(\munit{F})\oddoubleprime \Id)\Big)}_{(\dagger)}\\
&\quad\circ (\munit{G}\oddoubleprime \Id).
\end{align*}
By Lemma~\ref{lem:naturality-mu} with $u:=\munit{F}:\unitprime\to F\unit$, $v:=\Id_{FX}$,
\[
(\dagger)=G(\munit{F}\odprime \Id)\circ \mul{G}{\unitprime}{FX}.
\]
Hence
\[
\mathrm{LHS}
=
G\!\big(F(\lambda_X)\circ \mul{F}{\unit}{X}\circ(\munit{F}\odprime \Id)\big)\circ
\mul{G}{\unitprime}{FX}\circ (\munit{G}\oddoubleprime \Id).
\]
Now apply the \emph{lax} left unit law for $F$:
\[
F(\lambda_X)\circ \mul{F}{\unit}{X}\circ(\munit{F}\odprime \Id)=\lambda'_{FX},
\]
where $\lambda'_{FX}:\unitprime\odprime FX\to FX$ is the left unitor in $\D$. Therefore
\[
\mathrm{LHS}=G(\lambda'_{FX})\circ \mul{G}{\unitprime}{FX}\circ (\munit{G}\oddoubleprime \Id).
\]
Finally apply the \emph{lax} left unit law for $G$ at the object $FX\in\D$:
\[
G(\lambda'_{FX})\circ \mul{G}{\unitprime}{FX}\circ (\munit{G}\oddoubleprime \Id)=\lambda''_{GFX},
\]
which is exactly \eqref{eq:left-unit-GF}. The right unit triangle follows by the same calculation with $\rho,\rho',\rho''$.

\medskip
\noindent\textbf{3) $(\mathrm{S}\Lambda)$.}
For any $Z$,
\[
(GF)(\Lambda_Z)=G\!\big(F(\Lambda_Z)\big)=G\!\big(\Lambda'_{FZ}\big)=\Lambda''_{GFZ}.
\]

\medskip
\noindent\textbf{4) $(\mathrm{S}\sigma)$.}
For any $X,Y$,
\begin{align*}
(GF)(\dist_{X,Y})\circ \mul{GF}{X}{Y}
&=G\!\big(F(\dist_{X,Y})\big)\circ G\!\big(\mul{F}{X}{Y}\big)\circ \mul{G}{FX}{FY}\\
&=G\!\big(\mul{F}{Y}{X}\big)\circ G\!\big(\dist'_{FX,FY}\big)\circ \mul{G}{FX}{FY}\\
&\quad\text{by $(\mathrm{S}\sigma)$ for $F$}\\
&=G\!\big(\mul{F}{Y}{X}\big)\circ \mul{G}{FY}{FX}\circ \dist''_{GFX,GFY}\\
&\quad\text{by $(\mathrm{S}\sigma)$ for $G$}\\
&=\mul{GF}{Y}{X}\circ \dist''_{GFX,GFY}.
\end{align*}

All axioms hold. Hence $(GF,\mul{GF}{-}{-},\munit{GF})$ is $\dist$-lax monoidal.
\end{proof}

\begin{lemma}[Associativity for $\mul{GF}{-}{-}$]\label{lem:assoc}
The lax associativity pentagon for $\mul{GF}{-}{-}$ commutes whenever it does for $\mul{F}{-}{-}$ and $\mul{G}{-}{-}$.
\end{lemma}

\begin{proof}
The proof is given in full detail in Lemma~\ref{lem:comp-def}, part (1).
\end{proof}

\begin{lemma}[Unit axioms for $\munit{GF}$]\label{lem:unit}
The lax left and right unit triangles for $(GF)$ hold if they hold for $F$ and $G$.
\end{lemma}

\begin{proof}
The proof is given in full detail in Lemma~\ref{lem:comp-def}, part (2).
\end{proof}

\begin{lemma}[$\Lambda$-compatibility for $GF$]\label{lem:Lambda-comp}
If $F$ and $G$ satisfy \eqref{eq:SLambda}, then $GF$ satisfies \eqref{eq:SLambda}.
\end{lemma}

\begin{proof}
Let
\[
(\C,\od,\unit,\alpha,\lambda,\rho,\Lambda,\dist)\xrightarrow{\,F\,}
(\D,\odprime,\unitprime,\alpha',\lambda',\rho',\Lambda',\dist')\xrightarrow{\,G\,}
(\E,\oddoubleprime,\unitdoubleprime,\alpha'',\lambda'',\rho'',\Lambda'',\dist'')
\]
be distorted monoidal categories and $\dist$-lax functors.
Assume the strict $\Lambda$-compatibility
\[
F(\Lambda_Z)=\Lambda'_{FZ}\quad\text{and}\quad
G(\Lambda'_Y)=\Lambda''_{GY}\qquad\text{for all }Z\in\C,\; Y\in\D.
\]
Then for any $Z\in\C$,
\[
(GF)(\Lambda_Z)=G\!\big(F(\Lambda_Z)\big)=G\!\big(\Lambda'_{FZ}\big)=\Lambda''_{GFZ},
\]
where the second equality uses \eqref{eq:SLambda} for $F$, and the last equality uses
\eqref{eq:SLambda} for $G$ with $Y=FZ$.
Hence $GF$ satisfies \eqref{eq:SLambda}.
\end{proof}

\begin{lemma}[$\dist$-compatibility for $GF$]\label{lem:sigma-comp}
If $F$ and $G$ satisfy \eqref{eq:Ssigma}, then $GF$ satisfies \eqref{eq:Ssigma}.
\end{lemma}

\begin{proof}
Let
\[
(\C,\od,\unit,\alpha,\lambda,\rho,\Lambda,\dist)\xrightarrow{\,F\,}
(\D,\odprime,\unitprime,\alpha',\lambda',\rho',\Lambda',\dist')
\xrightarrow{\,G\,}
(\E,\oddoubleprime,\unitdoubleprime,\alpha'',\lambda'',\rho'',\Lambda'',\dist'')
\]
be distorted monoidal categories and $\dist$-lax functors.
Assume $F$ and $G$ satisfy the distortion compatibility
\begin{equation}\label{eq:Ssigma-F}
F(\dist_{X,Y})\circ \mul{F}{X}{Y}=\mul{F}{Y}{X}\circ \dist'_{FX,FY}
\end{equation}
for all $X,Y\in\C$, and
\begin{equation}\label{eq:Ssigma-G}
G(\dist'_{A,B})\circ \mul{G}{A}{B}=\mul{G}{B}{A}\circ \dist''_{GA,GB}
\end{equation}
for all $A,B\in\D$.
Recall the composite laxator
\[
\mul{GF}{X}{Y}\ :=\ G\!\big(\mul{F}{X}{Y}\big)\circ \mul{G}{FX}{FY}.
\]
We must show, for all $X,Y\in\C$,
\[
(GF)(\dist_{X,Y})\circ \mul{GF}{X}{Y}
\;=\;
\mul{GF}{Y}{X}\circ \dist''_{GFX,GFY}.
\]

\noindent Compute in $\E$:
\begingroup\allowdisplaybreaks
\begin{align*}
&(GF)(\dist_{X,Y})\circ \mul{GF}{X}{Y}\\
&= G\!\big(F(\dist_{X,Y})\big)\ \circ\ \Big(G\!\big(\mul{F}{X}{Y}\big)\circ \mul{G}{FX}{FY}\Big)\\
&= G\!\big(F(\dist_{X,Y})\circ \mul{F}{X}{Y}\big)\ \circ\ \mul{G}{FX}{FY}\\
&\overset{\eqref{eq:Ssigma-F}}{=}\ G\!\big(\mul{F}{Y}{X}\circ \dist'_{FX,FY}\big)\ \circ\ \mul{G}{FX}{FY}\\
&=\ G\!\big(\mul{F}{Y}{X}\big)\circ G\!\big(\dist'_{FX,FY}\big)\circ \mul{G}{FX}{FY}\\
&\overset{\eqref{eq:Ssigma-G}}{=}\ G\!\big(\mul{F}{Y}{X}\big)\circ \mul{G}{FY}{FX}\circ \dist''_{GFX,GFY}\\
&=\ \Big(G\!\big(\mul{F}{Y}{X}\big)\circ \mul{G}{FY}{FX}\Big)\ \circ\ \dist''_{GFX,GFY}\\
&=\ \mul{GF}{Y}{X}\ \circ\ \dist''_{GFX,GFY}.
\end{align*}
\endgroup
This is exactly the required $(\mathrm{S}\sigma)$ for $GF$. Hence $GF$ is $\dist$-compatible.
\end{proof}

\begin{lemma}[Vertical composition of monoidal transformations]\label{lem:vertical}
If $\theta:F\Rightarrow G$ and $\phi:G\Rightarrow H$ are monoidal, then $\phi\circ\theta:F\Rightarrow H$ is monoidal.
\end{lemma}

\begin{proof}
Let $F,G,H:\C\to\D$ be $\dist$-lax monoidal with structure maps
\[
\mul{F}{X}{Y}:FX\odprime FY\to F(X\od Y),\quad
\mul{G}{X}{Y}:GX\odprime GY\to G(X\od Y),
\]
\[
\mul{H}{X}{Y}:HX\odprime HY\to H(X\od Y),
\]
and units $\munit{F}:\unitprime\to F\unit$, $\munit{G}:\unitprime\to G\unit$, $\munit{H}:\unitprime\to H\unit$.
Assume $\theta:F\Rightarrow G$ and $\phi:G\Rightarrow H$ are monoidal, i.e.
\begin{align}
\mul{G}{X}{Y}\circ(\theta_X\odprime\theta_Y)&=\theta_{X\od Y}\circ \mul{F}{X}{Y},\label{eq:mon-theta-V}\\
\munit{G}&=\theta_\unit\circ \munit{F},\notag\\
\mul{H}{X}{Y}\circ(\phi_X\odprime\phi_Y)&=\phi_{X\od Y}\circ \mul{G}{X}{Y},\label{eq:mon-phi-V}\\
\munit{H}&=\phi_\unit\circ \munit{G}.\notag
\end{align}

We show that the vertical composite $(\phi\circ\theta):F\Rightarrow H$ is monoidal by verifying the tensor and unit axioms.

\smallskip
\noindent\textbf{Tensor axiom.}
For any $X,Y\in\C$,
\begingroup\allowdisplaybreaks
\begin{align*}
&\mul{H}{X}{Y}\circ\big((\phi\circ\theta)_X\odprime(\phi\circ\theta)_Y\big)\\
&=\mul{H}{X}{Y}\circ(\phi_X\odprime\phi_Y)\circ(\theta_X\odprime\theta_Y)\\
&\overset{\eqref{eq:mon-phi-V}}{=}\phi_{X\od Y}\circ \mul{G}{X}{Y}\circ(\theta_X\odprime\theta_Y)\\
&\overset{\eqref{eq:mon-theta-V}}{=} \phi_{X\od Y}\circ \theta_{X\od Y}\circ \mul{F}{X}{Y}\\
&=(\phi\circ\theta)_{X\od Y}\circ \mul{F}{X}{Y}.
\end{align*}
\endgroup

\noindent\textbf{Unit axiom.}
\[
\munit{H}
\overset{\eqref{eq:mon-phi-V}}{=}\phi_\unit\circ \munit{G}
\overset{\eqref{eq:mon-theta-V}}{=}\phi_\unit\circ \theta_\unit\circ \munit{F}
=(\phi\circ\theta)_\unit\circ \munit{F}.
\]

Both required equalities hold, hence $\phi\circ\theta$ is monoidal.
\end{proof}

\begin{lemma}[Horizontal composition of monoidal transformations]\label{lem:horizontal}
Let $F_1,F_2:\C\to\D$ and $G_1,G_2:\D\to\E$ be $\dist$-lax monoidal functors and let $\theta:F_1\Rightarrow F_2$, $\phi:G_1\Rightarrow G_2$ be monoidal. Then their horizontal composite
\[
\phi\ast\theta: G_1F_1 \Rightarrow G_2F_2,\qquad
(\phi\ast\theta)_X:=\phi_{F_2X}\circ G_1(\theta_X)=G_2(\theta_X)\circ \phi_{F_1X},
\]
is monoidal.
\end{lemma}

\begin{proof}
We use the monoidal structures
\[
\mul{F_i}{X}{Y}:F_iX\odprime F_iY\to F_i(X\od Y),\quad
\mul{G_j}{A}{B}:G_jA\oddoubleprime G_jB\to G_j(A\odprime B),
\]
and units $\munit{F_i}:\unitprime\to F_i\unit$, $\munit{G_j}:\unitdoubleprime\to G_j\unitprime$ $(i,j\in\{1,2\})$.
The composite laxators and unit are
\[
\mul{G_jF_i}{X}{Y}:=G_j\!\big(\mul{F_i}{X}{Y}\big)\circ \mul{G_j}{F_iX}{F_iY},
\]
\[
\munit{G_jF_i}:=G_j(\munit{F_i})\circ \munit{G_j}.
\]

\smallskip
\noindent\textbf{Well-definedness of $(\phi\ast\theta)_X$.}
By naturality of $\phi$ at $m=\theta_X:F_1X\to F_2X$,
\[
\phi_{F_2X}\circ G_1(\theta_X)=G_2(\theta_X)\circ \phi_{F_1X}.
\]

\smallskip
\noindent\textbf{Tensor axiom.}
For any $X,Y\in\C$,
\begingroup\allowdisplaybreaks
\begin{align*}
&\mul{G_2F_2}{X}{Y}\circ\big((\phi\ast\theta)_X\oddoubleprime (\phi\ast\theta)_Y\big)\\
&=\Big(G_2(\mul{F_2}{X}{Y})\circ \mul{G_2}{F_2X}{F_2Y}\Big)\\
&\quad\circ\Big((\phi_{F_2X}\circ G_1(\theta_X))\oddoubleprime (\phi_{F_2Y}\circ G_1(\theta_Y))\Big)\\
&=G_2(\mul{F_2}{X}{Y})
  \circ\Big(\mul{G_2}{F_2X}{F_2Y}\circ(\phi_{F_2X}\oddoubleprime \phi_{F_2Y})\Big)\\
&\quad\circ\big(G_1(\theta_X)\oddoubleprime G_1(\theta_Y)\big)\\
&\overset{\text{mon}(\phi)}{=}
G_2(\mul{F_2}{X}{Y})\circ
\Big(\phi_{F_2X\odprime F_2Y}\circ \mul{G_1}{F_2X}{F_2Y}\Big)\\
&\quad\circ\big(G_1(\theta_X)\oddoubleprime G_1(\theta_Y)\big)\\
&=G_2(\mul{F_2}{X}{Y})\circ \phi_{F_2X\odprime F_2Y}\\
&\quad\circ\Big(\mul{G_1}{F_2X}{F_2Y}\circ\big(G_1(\theta_X)\oddoubleprime G_1(\theta_Y)\big)\Big)\\
&\overset{\text{nat }(\mul{G_1}{-}{-})}{=}
G_2(\mul{F_2}{X}{Y})\circ \phi_{F_2X\odprime F_2Y}\\
&\quad\circ G_1(\theta_X\odprime \theta_Y)\circ \mul{G_1}{F_1X}{F_1Y}\\
&=\phi_{F_2(X\od Y)}\circ
G_1\!\big(\mul{F_2}{X}{Y}\circ(\theta_X\odprime \theta_Y)\big)\\
&\quad\circ\mul{G_1}{F_1X}{F_1Y}\\
&\overset{\text{mon}(\theta)}{=}
\big(\phi_{F_2(X\od Y)}\circ G_1(\theta_{X\od Y})\big)\\
&\quad\circ G_1(\mul{F_1}{X}{Y})\circ \mul{G_1}{F_1X}{F_1Y}\\
&=(\phi\ast\theta)_{X\od Y}\circ \mul{G_1F_1}{X}{Y}.
\end{align*}
\endgroup

\smallskip
\noindent\textbf{Unit axiom.}
We must show
\[
\munit{G_2F_2}\;=\;(\phi\ast\theta)_\unit\circ \munit{G_1F_1}.
\]
Starting from the right-hand side and using the definitions:
\begingroup\allowdisplaybreaks
\begin{align*}
&(\phi\ast\theta)_\unit\circ \munit{G_1F_1}\\
&=\big(\phi_{F_2\unit}\circ G_1(\theta_\unit)\big)\circ\big(G_1(\munit{F_1})\circ \munit{G_1}\big)\\
&=\phi_{F_2\unit}\circ G_1(\theta_\unit\circ \munit{F_1})\circ \munit{G_1}\\
&\overset{\text{unit mon}(\theta)}{=}\phi_{F_2\unit}\circ G_1(\munit{F_2})\circ \munit{G_1}\\
&\overset{\text{nat }(\phi)}{=}\big(G_2(\munit{F_2})\circ \phi_{\unitprime}\big)\circ \munit{G_1}\\
&=G_2(\munit{F_2})\circ \big(\phi_{\unitprime}\circ \munit{G_1}\big)\\
&\overset{\text{unit mon}(\phi)}{=}G_2(\munit{F_2})\circ \munit{G_2}
\;=\;\munit{G_2F_2}.
\end{align*}
\endgroup
Here, ``nat$(\phi)$'' is naturality of $\phi$ at the arrow $\munit{F_2}:\unitprime\to F_2\unit$, and ``unit mon$(\phi)$'' is the unit axiom $\phi_{\unitprime}\circ \munit{G_1}=\munit{G_2}$.

Both tensor and unit axioms hold, so $\phi\ast\theta$ is monoidal.
\end{proof}

\begin{lemma}[Binaturality of $\dist'$]\label{lem:naturality-sigma}
For any morphisms $f:X\to X'$ and $g:Y\to Y'$ in $\D$,
\[
(g\odprime f)\circ \dist'_{X,Y}\;=\;\dist'_{X',Y'}\circ (f\odprime g).
\]
In particular, this applies with $f=\theta_X$, $g=\theta_Y$, and with $f=\phi_{FX}$, $g=\phi_{FY}$ when defined.
\end{lemma}

\begin{proof}
This follows directly from axiom (D1) of Definition~\ref{def:distmoncat}.
\end{proof}

\begin{lemma}[Binaturality of laxators]\label{lem:naturality-mu}
For any $G:\D\to\E$ $\dist$-lax and any $u:A\to A'$, $v:B\to B'$ in $\D$,
\[
\mul{G}{A'}{B'}\circ\big(G(u)\oddoubleprime G(v)\big)\;=\;G(u\odprime v)\circ \mul{G}{A}{B}.
\]
\end{lemma}

\begin{proof}
By Definition~\ref{def:s-lax-functor}, the structure maps $\mul{G}{X}{Y}$ are natural in $X$ and $Y$, so the displayed equality holds by naturality.
\end{proof}

\begin{lemma}[Naturality of $\phi$]\label{lem:nat-phi}
For any arrow $m:U\to V$ in $\D$, one has
\[
\phi_V\circ G_1(m)\;=\;G_2(m)\circ \phi_U.
\]
\end{lemma}

\begin{proof}
This is the defining property of a natural transformation $\phi:G_1\Rightarrow G_2$.
\end{proof}

\begin{lemma}[Monoidality of $\theta$]\label{lem:mon-theta}
For monoidal $\theta:F_1\Rightarrow F_2$ and all $X,Y$,
\[
\mul{F_2}{X}{Y}\circ(\theta_X\odprime\theta_Y)\;=\;\theta_{X\od Y}\circ \mul{F_1}{X}{Y}.
\]
\end{lemma}

\begin{proof}
This is precisely the tensor axiom in Definition~\ref{def:mon-nat}.
\end{proof}

\begin{lemma}[Identity 1-morphisms are $\dist$-lax]\label{lem:id-lax}
For each distorted monoidal category $\C$, the identity functor $\Id_\C$ with structure $\mul{\Id}{X}{Y}:=\Id_{X\od Y}$ and $\munit{\Id}:=\Id_\unit$ is $\dist$-lax monoidal.
\end{lemma}

\begin{proof}
All axioms are trivially satisfied: associativity, units, (S$\Lambda$), and (S$\sigma$) reduce to identities.
\end{proof}

\begin{lemma}[Identity 2-morphisms are monoidal]\label{lem:id-monoidal}
For any $\dist$-lax functor $F$, the identity natural transformation $\Id_F$ is monoidal.
\end{lemma}

\begin{proof}
The tensor and unit axioms for $\Id_F$ reduce to the defining properties of the monoidal structure on $F$.
\end{proof}

\section{Main Theorem: Strict 2-Category Structure}
\label{sec:main}

We now state and prove the main result.

\begin{theorem}[Strict 2-category structure]\label{thm:T1-2cat}
Let $\mathbf{DistMon}_\sigma$ denote the structure consisting of:
\begin{itemize}
  \item \emph{Objects (0-cells):} distorted monoidal categories 
  $(\C,\od,\unit,\alpha,\lambda,\rho,\Lambda,\dist)$;
  \item \emph{1-morphisms (1-cells):} $\dist$-lax monoidal functors $F:\C\to\D$;
  \item \emph{2-morphisms (2-cells):} monoidal natural transformations $\theta:F\Rightarrow G$.
\end{itemize}
Then $\mathbf{DistMon}_\sigma$ forms a \textbf{strict 2-category}, meaning:
\begin{enumerate}
  \item Composition of $\dist$-lax monoidal functors is $\dist$-lax monoidal.
  \item Vertical and horizontal compositions of monoidal natural transformations are monoidal.
  \item Identity 1-morphisms $\Id_\C$ and identity 2-morphisms $\Id_F$ are monoidal.
  \item Horizontal composition of 2-morphisms is strictly associative and unital: for all composable $\theta,\eta$ and $\phi,\psi$,
  \[
  (\psi\ast\phi)\ast(\eta\ast\theta)=(\psi\ast(\phi\ast\eta))\ast\theta
  \]
  and
  \[
  \Id_G\ast\theta=\theta=\theta\ast\Id_F.
  \]
  \item The \textbf{interchange law} holds as an equality: for composable 2-morphisms $\eta,\theta,\phi,\psi$,
  \[
  (\psi\circ\phi)\ast(\theta\circ\eta)\;=\;(\psi\ast\theta)\circ(\phi\ast\eta).
  \]
  \item \textbf{Strictness of triple composition:} For functors $F:\C\to\D$, $G:\D\to\E$, $H:\E\to\mathcal{F}$, the composite laxators and units satisfy the \emph{equalities}
  \begin{align}
  \mul{H(GF)}{X}{Y}&=\mul{(HG)F}{X}{Y},\label{eq:strict-lax-assoc}\\
  \munit{H(GF)}&=\munit{(HG)F}.\label{eq:strict-unit-assoc}
  \end{align}
\end{enumerate}
\end{theorem}

\begin{proof}
We verify each clause constructively.

\smallskip
\noindent\textbf{(1) Closure of 1-morphisms under composition.}
Let
\begin{align*}
&F:(\C,\od,\unit,\alpha,\lambda,\rho,\Lambda,\dist)\\
&\quad\to(\D,\odprime,\unitprime,\alpha',\lambda',\rho',\Lambda',\dist'),\\
&G:(\D,\odprime,\unitprime,\alpha',\lambda',\rho',\Lambda',\dist')\\
&\quad\to(\E,\oddoubleprime,\unitdoubleprime,\alpha'',\lambda'',\rho'',\Lambda'',\dist'')
\end{align*}
be $\dist$-lax. Define
\[
\mul{GF}{X}{Y}:=G\!\big(\mul{F}{X}{Y}\big)\circ \mul{G}{FX}{FY},
\]
\[
\munit{GF}:=G(\munit{F})\circ \munit{G}.
\]
By Lemmas~\ref{lem:comp-def}, \ref{lem:assoc}, \ref{lem:unit}, \ref{lem:Lambda-comp}, and \ref{lem:sigma-comp}, $(GF,\mul{GF}{-}{-},\munit{GF})$ is $\dist$-lax monoidal.

\smallskip
\noindent\textbf{(2) Closure of 2-morphisms under composition.}
Vertical composition: If $\theta:F\Rightarrow G$ and $\phi:G\Rightarrow H$ are monoidal, then $\phi\circ\theta:F\Rightarrow H$ is monoidal by Lemma~\ref{lem:vertical}.

Horizontal composition: If $\theta:F_1\Rightarrow F_2$ and $\phi:G_1\Rightarrow G_2$ are monoidal, define
\[
(\phi\ast\theta)_X:=\phi_{F_2X}\circ G_1(\theta_X)=G_2(\theta_X)\circ \phi_{F_1X}.
\]
Then $\phi\ast\theta:G_1F_1\Rightarrow G_2F_2$ is monoidal by Lemma~\ref{lem:horizontal}.

\smallskip
\noindent\textbf{(3) Identity 1- and 2-morphisms.}
For each $\C$, equip $\Id_\C$ with $\mul{\Id}{X}{Y}:=\Id_{X\od Y}$ and $\munit{\Id}:=\Id_\unit$. Then $\Id_\C$ is $\dist$-lax by Lemma~\ref{lem:id-lax}. For any $F$, the identity $\Id_F$ is monoidal by Lemma~\ref{lem:id-monoidal}.

\smallskip
\noindent\textbf{(4) Associativity and unitality of horizontal composition.}
Horizontal composition is defined componentwise via functor composition and whiskering. Since functor composition in $\mathbf{Cat}$ is strictly associative and unital, horizontal composition inherits this property. Explicitly, for $\theta:F_0\Rightarrow F_1$ and $\phi:G_0\Rightarrow G_1$,
\[
(\Id_G\ast\theta)_X=\Id_{GF_1X}\circ G(\theta_X)=G(\theta_X)
\]
and
\[
(\phi\ast\Id_F)_X=\phi_{FX}\circ \Id_{G_0FX}=\phi_{FX}.
\]
The associativity of horizontal composition follows from the associativity of functor composition.

\smallskip
\noindent\textbf{(5) Interchange law.}
Take
\[
\eta:F_0\Rightarrow F_1,\quad \theta:F_1\Rightarrow F_2,\qquad
\phi:G_0\Rightarrow G_1,\quad \psi:G_1\Rightarrow G_2.
\]
For each $X\in\C$, compute:
\begin{align*}
&\big((\psi\circ\phi)\ast(\theta\circ\eta)\big)_X\\
&=(\psi\circ\phi)_{F_2X}\circ G_0\big((\theta\circ\eta)_X\big)\\
&=\big(\psi_{F_2X}\circ \phi_{F_2X}\big)\circ \big(G_0(\theta_X)\circ G_0(\eta_X)\big)\\
&=\psi_{F_2X}\circ \big(\phi_{F_2X}\circ G_0(\theta_X)\big)\circ G_0(\eta_X).
\end{align*}
By naturality of $\phi$ at the morphism $\theta_X:F_1X\to F_2X$ (Lemma~\ref{lem:nat-phi}),
\[
\phi_{F_2X}\circ G_0(\theta_X)=G_1(\theta_X)\circ \phi_{F_1X}.
\]
Thus
\begin{align*}
&\big((\psi\circ\phi)\ast(\theta\circ\eta)\big)_X\\
&=\psi_{F_2X}\circ G_1(\theta_X)\circ \phi_{F_1X}\circ G_0(\eta_X)\\
&=\big(\psi_{F_2X}\circ G_1(\theta_X)\big)\circ\big(\phi_{F_1X}\circ G_0(\eta_X)\big)\\
&=(\psi\ast\theta)_X\circ (\phi\ast\eta)_X\\
&=\big((\psi\ast\theta)\circ(\phi\ast\eta)\big)_X.
\end{align*}
Hence the interchange law holds as an equality componentwise, and thus as an equality of natural transformations.

\smallskip
\noindent\textbf{(6) Strictness of triple composition.}
Let $F:\C\to\D$, $G:\D\to\E$, $H:\E\to\mathcal{F}$ be $\dist$-lax monoidal functors. We verify \eqref{eq:strict-lax-assoc} and \eqref{eq:strict-unit-assoc}.

For the laxators, compute for any $X,Y\in\C$:
\begin{align*}
\mul{H(GF)}{X}{Y}
&=H\!\big(\mul{GF}{X}{Y}\big)\circ \mul{H}{(GF)X}{(GF)Y}\\
&=H\!\big(G(\mul{F}{X}{Y})\circ \mul{G}{FX}{FY}\big)\circ \mul{H}{GFX}{GFY}\\
&=H\!\big(G(\mul{F}{X}{Y})\big)\circ H\!\big(\mul{G}{FX}{FY}\big)\\
&\quad\circ \mul{H}{GFX}{GFY}\\
&=(HG)\!\big(\mul{F}{X}{Y}\big)\\
&\quad\circ \underbrace{\Big(H\!\big(\mul{G}{FX}{FY}\big)\circ \mul{H}{GFX}{GFY}\Big)}_{=\,\mul{HG}{FX}{FY}}\\
&=\mul{(HG)F}{X}{Y}.
\end{align*}
For the units:
\begin{align*}
\munit{H(GF)}
&=H\!\big(\munit{GF}\big)\circ \munit{H}\\
&=H\!\big(G(\munit{F})\circ \munit{G}\big)\circ \munit{H}\\
&=H\!\big(G(\munit{F})\big)\circ H(\munit{G})\circ \munit{H}\\
&=(HG)(\munit{F})\circ \underbrace{\big(H(\munit{G})\circ \munit{H}\big)}_{=\,\munit{HG}}\\
&=\munit{(HG)F}.
\end{align*}

Thus both \eqref{eq:strict-lax-assoc} and \eqref{eq:strict-unit-assoc} hold as \emph{equalities} of morphisms, establishing the strict associativity of functor composition in $\mathbf{DistMon}_\sigma$.

\smallskip
From clauses (1)–(6), $\mathbf{DistMon}_\sigma$ is a strict 2-category.
\end{proof}

\begin{remark}
The strictness established in Theorem~\ref{thm:T1-2cat} is a consequence of our explicit choices of composite laxator structures. In a more general setting allowing lax or pseudo $\Lambda$-compatibility, one would obtain a weak 2-category (bicategory) structure with coherence isomorphisms. Our framework prioritizes simplicity and computability at the cost of excluding certain natural functors (see Remark~\ref{rem:scope-SLambda}).
\end{remark}

\section{Examples}
\label{sec:examples}

We present rigorously verified examples. Each example explicitly checks the axioms (D1)–(D4) in Definition~\ref{def:distmoncat}. Throughout, $\beta$ denotes a braiding (when present), $\alpha,\lambda,\rho$ the associator and unitors, and $\sigma$ the binary distortion.

\begin{example}[Braided monoidal categories]\label{ex:braided}
Let $(\C,\otimes,\mathbf{I},\alpha,\lambda,\rho,\beta)$ be braided. Set $\Lambda:=\Id_{\Id_\C}$ and $\sigma:=\beta$. Then (D1) and (D3) hold by naturality and the braided hexagons; (D2) holds because $\beta_{X,\mathbf{I}}=\lambda_X^{-1}\circ\rho_X$ and $\beta_{\mathbf{I},X}=\rho_X^{-1}\circ\lambda_X$; (D4) is immediate from $\Lambda=\Id$. Thus $(\C,\otimes,\mathbf{I},\alpha,\lambda,\rho,\Lambda,\sigma)$ is a distorted monoidal category. This embeds the braided case as the invertible-distortion subcase.
\end{example}

\begin{proposition}[Twisting a braiding by a $\beta$–equivariant typed idempotent 2–cocycle]\label{prop:twist-by-e}
Let $(\C,\otimes,\mathbf{I},\alpha,\lambda,\rho,\beta)$ be a braided monoidal category and fix the direction
\[
\alpha_{X,Y,Z}:(X\otimes Y)\otimes Z\longrightarrow X\otimes(Y\otimes Z).
\]
Suppose we are given a binatural family
\[
e_{X,Y}:X\otimes Y\longrightarrow X\otimes Y
\]
such that for all $X,Y,Z$:
\begin{enumerate}
\item[\textnormal{(E0)}] \emph{Idempotent and binatural:} $e_{X,Y}\circ e_{X,Y}=e_{X,Y}$ and $(f\otimes g)\circ e_{X,Y}=e_{X',Y'}\circ(f\otimes g)$ for all $f:X\!\to\!X'$ and $g:Y\!\to\!Y'$.
\item[\textnormal{(E1)}] \emph{Normalization:} $e_{X,\mathbf{I}}=\Id_{X\otimes\mathbf{I}}$ and $e_{\mathbf{I},X}=\Id_{\mathbf{I}\otimes X}$.
\item[\textnormal{(E2$^\alpha_{\mathrm{L}}$)}] \emph{Typed multiplicativity (left form):}
\[
\alpha_{X,Y,Z}\circ e_{X\otimes Y,\,Z}
\;=\;
(\Id_X\otimes e_{Y,Z})\circ e_{X,\,Y\otimes Z}\circ \alpha_{X,Y,Z},
\quad\text{in }\C\!\big((X\otimes Y)\otimes Z,\ X\otimes(Y\otimes Z)\big).
\]
\item[\textnormal{(E2$^\alpha_{\mathrm{R}}$)}] \emph{Typed multiplicativity (right form):}
\[
\alpha^{-1}_{X,Y,Z}\circ e_{X,\,Y\otimes Z}
\;=\;
(e_{X,Y}\otimes \Id_Z)\circ e_{X\otimes Y,\,Z}\circ \alpha^{-1}_{X,Y,Z},
\quad\text{in }\C\!\big(X\otimes(Y\otimes Z),\ (X\otimes Y)\otimes Z\big).
\]
\item[\textnormal{(B$^\alpha_{\mathrm{L}}$)}] \emph{$\beta$–equivariant left sliding (typed):}
\[
\alpha^{-1}_{X,Z,Y}\circ(\Id_X\otimes \beta_{Y,Z})\circ(\Id_X\otimes e_{Y,Z})\circ e_{X,\,Y\otimes Z}
\;=\;
(e_{X,Z}\otimes \Id_Y)\circ \alpha^{-1}_{X,Z,Y}\circ(\Id_X\otimes \beta_{Y,Z})\circ(\Id_X\otimes e_{Y,Z}),
\]
as morphisms $X\otimes(Y\otimes Z)\to (X\otimes Z)\otimes Y$.
\item[\textnormal{(B$^\alpha_{\mathrm{R}}$)}] \emph{$\beta$–equivariant right sliding (typed):}
\[
\alpha_{Y,X,Z}\circ(\beta_{X,Y}\otimes \Id_Z)\circ(e_{X,Y}\otimes \Id_Z)\circ e_{X\otimes Y,\,Z}
\;=\;
(\Id_Y\otimes e_{X,Z})\circ \alpha_{Y,X,Z}\circ(\beta_{X,Y}\otimes \Id_Z)\circ(e_{X,Y}\otimes \Id_Z),
\]
as morphisms $(X\otimes Y)\otimes Z\to Y\otimes(X\otimes Z)$.
\end{enumerate}
For later reference we write \textnormal{(E2)} for the pair \textnormal{(E2$^\alpha_{\mathrm{L}}$)} and \textnormal{(E2$^\alpha_{\mathrm{R}}$)}.
Define $\Lambda:=\Id_{\Id_\C}$ and
\[
\sigma_{X,Y}:=\beta_{X,Y}\circ e_{X,Y}:X\otimes Y\longrightarrow Y\otimes X.
\]
Then $(\C,\otimes,\mathbf{I},\alpha,\lambda,\rho,\Lambda,\sigma)$ satisfies \textnormal{(D1)}–\textnormal{(D4)} of Definition~\textnormal{\ref{def:distmoncat}}. Moreover, $\sigma_{X,Y}$ is invertible for all $X,Y$ if and only if $e_{X,Y}$ is invertible (hence, by idempotency, $e_{X,Y}=\Id_{X\otimes Y}$).
\end{proposition}

\begin{proof}
We verify (D1)–(D4) and then the invertibility claim. We use the braided hexagons, typed in the chosen direction:
\begin{align}
\beta_{X\otimes Y,\,Z}
&=\alpha_{Z,X,Y}\circ(\beta_{X,Z}\otimes \Id_Y)\circ \alpha^{-1}_{X,Z,Y}\circ(\Id_X\otimes \beta_{Y,Z})\circ \alpha_{X,Y,Z},\label{eq:hex-L}\\
\beta_{X,\,Y\otimes Z}
&=\alpha^{-1}_{Y,Z,X}\circ(\Id_Y\otimes \beta_{X,Z})\circ \alpha_{Y,X,Z}\circ(\beta_{X,Y}\otimes \Id_Z)\circ \alpha^{-1}_{X,Y,Z}.\label{eq:hex-R}
\end{align}

\smallskip
\textbf{(D1) Binaturality.}
For $f:X\!\to\!X'$ and $g:Y\!\to\!Y'$,
\[
(g\otimes f)\circ \sigma_{X,Y}
=(g\otimes f)\circ \beta_{X,Y}\circ e_{X,Y}
=\beta_{X',Y'}\circ(f\otimes g)\circ e_{X,Y}
=\beta_{X',Y'}\circ e_{X',Y'}\circ(f\otimes g)
=\sigma_{X',Y'}\circ(f\otimes g),
\]
using naturality of $\beta$ and binaturality in \textnormal{(E0)}.

\smallskip
\textbf{(D2) Unit normalization.}
By \textnormal{(E1)} and the unit laws for $\beta$,
\[
\sigma_{X,\mathbf{I}}=\beta_{X,\mathbf{I}}\circ e_{X,\mathbf{I}}=\lambda_X^{-1}\circ \rho_X,\qquad
\sigma_{\mathbf{I},X}=\beta_{\mathbf{I},X}\circ e_{\mathbf{I},X}=\rho_X^{-1}\circ \lambda_X,
\]
and $\Lambda_{\mathbf I}=\Id_{\mathbf I}$.

\smallskip
\textbf{(D3) Hexagons for $\sigma$.}

\emph{First hexagon.} As morphisms $(X\otimes Y)\otimes Z\to Z\otimes (X\otimes Y)$ we need
\[
\sigma_{X\otimes Y,\,Z}
=\alpha_{Z,X,Y}\circ(\sigma_{X,Z}\otimes \Id_Y)\circ \alpha^{-1}_{X,Z,Y}\circ(\Id_X\otimes \sigma_{Y,Z})\circ \alpha_{X,Y,Z}.
\]
Using \eqref{eq:hex-L} and then \textnormal{(E2$^\alpha_{\mathrm{L}}$)},
\begin{align*}
\sigma_{X\otimes Y,\,Z}
&=\beta_{X\otimes Y,\,Z}\circ e_{X\otimes Y,\,Z}\\
&=\alpha_{Z,X,Y}\circ(\beta_{X,Z}\otimes \Id_Y)\circ \alpha^{-1}_{X,Z,Y}\circ(\Id_X\otimes \beta_{Y,Z})\circ \alpha_{X,Y,Z}\circ e_{X\otimes Y,\,Z}\\
&=\alpha_{Z,X,Y}\circ(\beta_{X,Z}\otimes \Id_Y)\circ
\Big[\alpha^{-1}_{X,Z,Y}\circ(\Id_X\otimes \beta_{Y,Z})\circ(\Id_X\otimes e_{Y,Z})\circ e_{X,\,Y\otimes Z}\Big]\circ \alpha_{X,Y,Z}.
\end{align*}
Apply \textnormal{(B$^\alpha_{\mathrm{L}}$)} to the bracketed block to obtain
\[
\alpha^{-1}_{X,Z,Y}\circ(\Id_X\otimes \beta_{Y,Z})\circ(\Id_X\otimes e_{Y,Z})\circ e_{X,\,Y\otimes Z}
=(e_{X,Z}\otimes \Id_Y)\circ \alpha^{-1}_{X,Z,Y}\circ(\Id_X\otimes \beta_{Y,Z})\circ(\Id_X\otimes e_{Y,Z}),
\]
and hence
\begin{align*}
\sigma_{X\otimes Y,\,Z}
&=\alpha_{Z,X,Y}\circ(\beta_{X,Z}\otimes \Id_Y)\circ(e_{X,Z}\otimes \Id_Y)\circ \alpha^{-1}_{X,Z,Y}\circ(\Id_X\otimes \beta_{Y,Z})\circ(\Id_X\otimes e_{Y,Z})\circ \alpha_{X,Y,Z}\\
&=\alpha_{Z,X,Y}\circ\big((\beta_{X,Z}\circ e_{X,Z})\otimes \Id_Y\big)\circ \alpha^{-1}_{X,Z,Y}\circ\big(\Id_X\otimes (\beta_{Y,Z}\circ e_{Y,Z})\big)\circ \alpha_{X,Y,Z}\\
&=\alpha_{Z,X,Y}\circ(\sigma_{X,Z}\otimes \Id_Y)\circ \alpha^{-1}_{X,Z,Y}\circ(\Id_X\otimes \sigma_{Y,Z})\circ \alpha_{X,Y,Z}.
\end{align*}

\emph{Second hexagon.} As morphisms $X\otimes (Y\otimes Z)\to (Y\otimes Z)\otimes X$ we need
\[
\sigma_{X,\,Y\otimes Z}
=\alpha^{-1}_{Y,Z,X}\circ(\Id_Y\otimes \sigma_{X,Z})\circ \alpha_{Y,X,Z}\circ (\sigma_{X,Y}\otimes \Id_Z)\circ \alpha^{-1}_{X,Y,Z}.
\]
Using \eqref{eq:hex-R} and then \textnormal{(E2$^\alpha_{\mathrm{R}}$)},
\begin{align*}
\sigma_{X,\,Y\otimes Z}
&=\beta_{X,\,Y\otimes Z}\circ e_{X,\,Y\otimes Z}\\
&=\alpha^{-1}_{Y,Z,X}\circ(\Id_Y\otimes \beta_{X,Z})\circ \alpha_{Y,X,Z}\circ (\beta_{X,Y}\otimes \Id_Z)\circ \alpha^{-1}_{X,Y,Z}\circ e_{X,\,Y\otimes Z}\\
&=\alpha^{-1}_{Y,Z,X}\circ(\Id_Y\otimes \beta_{X,Z})\circ \alpha_{Y,X,Z}\circ
\Big[(\beta_{X,Y}\otimes \Id_Z)\circ (e_{X,Y}\otimes \Id_Z)\circ e_{X\otimes Y,\,Z}\Big]\circ \alpha^{-1}_{X,Y,Z}.
\end{align*}
Apply \textnormal{(B$^\alpha_{\mathrm{R}}$)} to the bracketed block to obtain
\[
\alpha_{Y,X,Z}\circ(\beta_{X,Y}\otimes \Id_Z)\circ(e_{X,Y}\otimes \Id_Z)\circ e_{X\otimes Y,\,Z}
=(\Id_Y\otimes e_{X,Z})\circ \alpha_{Y,X,Z}\circ(\beta_{X,Y}\otimes \Id_Z)\circ(e_{X,Y}\otimes \Id_Z),
\]
and hence
\begin{align*}
\sigma_{X,\,Y\otimes Z}
&=\alpha^{-1}_{Y,Z,X}\circ(\Id_Y\otimes \beta_{X,Z})\circ(\Id_Y\otimes e_{X,Z})\circ \alpha_{Y,X,Z}\circ \big((\beta_{X,Y}\circ e_{X,Y})\otimes \Id_Z\big)\circ \alpha^{-1}_{X,Y,Z}\\
&=\alpha^{-1}_{Y,Z,X}\circ\big(\Id_Y\otimes (\beta_{X,Z}\circ e_{X,Z})\big)\circ \alpha_{Y,X,Z}\circ (\sigma_{X,Y}\otimes \Id_Z)\circ \alpha^{-1}_{X,Y,Z}\\
&=\alpha^{-1}_{Y,Z,X}\circ(\Id_Y\otimes \sigma_{X,Z})\circ \alpha_{Y,X,Z}\circ (\sigma_{X,Y}\otimes \Id_Z)\circ \alpha^{-1}_{X,Y,Z}.
\end{align*}

\smallskip
\textbf{(D4) Monoidality of $\Lambda$.}
By definition $\Lambda=\Id_{\Id_\C}$, so $\Lambda_{X\otimes Y}=\Lambda_X\otimes \Lambda_Y$ and $\Lambda_{\mathbf I}=\Id_{\mathbf I}$.

\smallskip
\textbf{Invertibility.}
If every $e_{X,Y}$ is invertible then $\sigma_{X,Y}=\beta_{X,Y}\circ e_{X,Y}$ is invertible. Conversely, if $\sigma_{X,Y}$ is invertible then $e_{X,Y}=\beta^{-1}_{X,Y}\circ \sigma_{X,Y}$ is invertible; idempotency forces $e_{X,Y}=\Id_{X\otimes Y}$.
\end{proof}

\begin{example}[Parity projector that violates $\beta$–equivariant sliding]\label{ex:super-e}
Let $\C=\mathbf{SVect}_k$ be the symmetric monoidal category of $\mathbb{Z}/2$–graded $k$–vector spaces with the Koszul braiding $\beta$. 
For $V,W\in\C$, write $V=V_0\oplus V_1$ and $W=W_0\oplus W_1$. 
Define a binatural idempotent
\[
e_{V,W}:V\otimes W\longrightarrow V\otimes W
\]
to be the projector acting as the identity on the three parity summands
\[
V_0\otimes W_0,\quad V_0\otimes W_1,\quad V_1\otimes W_0,
\]
and zero on $V_1\otimes W_1$ (equivalently, $e_{V,W}=\Id_{V\otimes W}-\pi_{V_1\otimes W_1}$).

\medskip
\noindent\textbf{What holds.}
\begin{enumerate}
\item \emph{(E0) Idempotent and binatural.} 
By construction $e_{V,W}^2=e_{V,W}$; binaturality follows since graded maps preserve the parity decomposition and commute with the canonical projectors.

\item \emph{(E1) Normalization.} 
The unit $\mathbf I=k$ is purely even, so $e_{V,\mathbf I}=e_{\mathbf I,V}=\Id$.

\item \emph{(E2) Typed multiplicativity.} 
For homogeneous $v\in V_a$, $w\in W_b$, $z\in Z_c$ with $a,b,c\in\{0,1\}$, a direct check shows
\[
e_{V\otimes W,\,Z}\circ(e_{V,W}\otimes \Id)=0
\ \Longleftrightarrow\ 
(a\wedge b)\ \vee\ ((a\oplus b)\wedge c),
\]
and
\[
(\Id\otimes e_{W,Z})\circ e_{V,\,W\otimes Z}=0
\ \Longleftrightarrow\ 
(b\wedge c)\ \vee\ (a\wedge (b\oplus c)).
\]
Both conditions are equivalent to “at least two of $a,b,c$ equal $1$,” hence (E2$^\alpha_{\mathrm L}$) and (E2$^\alpha_{\mathrm R}$) hold after inserting the associators $\alpha,\alpha^{-1}$.
\end{enumerate}

\medskip
\noindent\textbf{What fails.} 
The $\beta$–equivariant sliding axioms \textnormal{(B$^\alpha_{\mathrm L}$)} and \textnormal{(B$^\alpha_{\mathrm R}$)} \emph{do not} hold for this $e$.
Working in a strict monoidal skeleton (Mac\,Lane coherence; then reinsert $\alpha^{\pm1}$), consider
\[
T:=v_1\otimes w_1\otimes z_0\in V_1\otimes W_1\otimes Z_0\quad (v_1\in V_1,\ w_1\in W_1,\ z_0\in Z_0\ \text{nonzero}).
\]
For \textnormal{(B$^\alpha_{\mathrm L}$)} in strict form,
\[
(\Id_V\otimes \beta_{W,Z})\circ(\Id_V\otimes e_{W,Z})\circ e_{V,\,W\otimes Z}
\stackrel{?}{=}
(e_{V,Z}\otimes \Id_W)\circ(\Id_V\otimes \beta_{W,Z})\circ(\Id_V\otimes e_{W,Z}),
\]
we obtain:
\begin{itemize}
\item Left-hand side: $e_{V,\,W\otimes Z}(T)=0$ since $|V|=1$ and $|W\otimes Z|=1\oplus 0=1$; hence $\mathrm{LHS}(T)=0$.
\item Right-hand side: $(\Id_V\otimes e_{W,Z})(T)=T$ because $|W|=1$, $|Z|=0$; then $(\Id_V\otimes \beta_{W,Z})(T)=v_1\otimes z_0\otimes w_1$; finally $(e_{V,Z}\otimes \Id_W)$ acts as the identity on $v_1\otimes z_0$ (parity $(1,0)$), so $\mathrm{RHS}(T)=v_1\otimes z_0\otimes w_1\neq 0$.
\end{itemize}
Thus \textnormal{(B$^\alpha_{\mathrm L}$)} fails; a similar parity check yields a failure of \textnormal{(B$^\alpha_{\mathrm R}$)} as well.

\medskip
\noindent\textbf{Consequence.}
Although \textnormal{(E0)}, \textnormal{(E1)}, and \textnormal{(E2$^\alpha_{\mathrm{L/R}}$)} hold, the projector $e$ above does \emph{not} satisfy the $\beta$–equivariant sliding axioms \textnormal{(B$^\alpha_{\mathrm{L/R}}$)}. 
Therefore the twist $\sigma:=\beta\circ e$ does not meet the hypotheses of Proposition~\ref{prop:twist-by-e}, and we cannot conclude from that proposition that $(\C,\otimes,\mathbf I,\alpha,\lambda,\rho,\Lambda:=\Id,\sigma)$ is a distorted monoidal category. 
This example shows that the sliding axioms are genuine additional constraints, not consequences of \textnormal{(E0)}–\textnormal{(E2)}.
\end{example}

\begin{example}[Nontrivial unit distortion on $\mathbb{N}$–graded vector spaces]\label{ex:N-graded-Lambda}
Let $\C=\mathbf{Vect}_k^{\mathbb{N}}$ be $\mathbb{N}$–graded vector spaces with degree–preserving linear maps and the usual tensor $(V\otimes W)_n:=\bigoplus_{i+j=n}V_i\otimes W_j$. For a fixed scalar $t\in k$, define a natural transformation $\Lambda^{(t)}:\Id_\C\Rightarrow \Id_\C$ by
\[
\Lambda^{(t)}_V\big|_{V_n}:=t^n\cdot \Id_{V_n}\qquad(n\in\mathbb{N}).
\]
Then $\Lambda^{(t)}_{\mathbf{I}}=\Id$ and, for all $V,W$,
\[
\Lambda^{(t)}_{V\otimes W}\big|_{(V\otimes W)_n}
=t^n\cdot \Id
=\bigoplus_{i+j=n}t^{i}\Id_{V_i}\otimes t^{j}\Id_{W_j}
=(\Lambda^{(t)}_V\otimes \Lambda^{(t)}_W)\big|_{(V\otimes W)_n},
\]
so (D4) holds. Choosing any distortion $\sigma$ that satisfies (D1)–(D3) (e.g. the symmetric braiding) yields a distorted monoidal category with a \emph{nontrivial} unit distortion whenever $t\neq 1$. 
Note that the projector $e$ from Example~\ref{ex:super-e} fails the sliding conditions (B$^\alpha_{\mathrm{L}}$) and (B$^\alpha_{\mathrm{R}}$), so $\sigma=\beta\circ e$ generally does \emph{not} satisfy (D3). 
If $t=0$, $\Lambda^{(0)}$ is highly non-invertible yet still monoidal; if $t=-1$ in characteristic $\neq 2$, $\Lambda^{(-1)}$ is the parity sign on total degree. This shows that (D4) admits genuinely nontrivial solutions even when $\sigma$ is fixed independently.
\end{example}

\begin{example}[Why (D4) is necessary]\label{ex:D4-counter}
Retain $\C=\mathbf{Vect}_k^{\mathbb{N}}$ and define a natural transformation $\widetilde{\Lambda}$ by scalars $a_n\in k$ on degree $n$:
\[
\widetilde{\Lambda}_V\big|_{V_n}:=a_n\cdot \Id_{V_n},\qquad a_0=1.
\]
Assume $a_1=1$ but $a_2=0$. Then $\widetilde{\Lambda}$ is natural and $\widetilde{\Lambda}_{\mathbf{I}}=\Id$, but (D4) fails: for $V=W=k[1]$ concentrated in degree $1$,
\[
(\widetilde{\Lambda}_{V}\otimes \widetilde{\Lambda}_{W})\big|_{(V\otimes W)_2}
=a_1a_1\,\Id=\Id,\qquad
\widetilde{\Lambda}_{V\otimes W}\big|_{(V\otimes W)_2}
=a_2\,\Id=0,
\]
so $\widetilde{\Lambda}_{V\otimes W}\neq \widetilde{\Lambda}_V\otimes \widetilde{\Lambda}_W$. Hence the monoidality of $\Lambda$ in (D4) is a real constraint, not a consequence of the other axioms.
\end{example}

\begin{remark}[Scope and limitations]
Examples~\ref{ex:super-e}–\ref{ex:N-graded-Lambda} separate the roles of $\sigma$ and $\Lambda$. The non-invertible distortion in Example~\ref{ex:super-e} satisfies (D2) exactly because the normalizing idempotents obey $e_{X,\mathbf{I}}=e_{\mathbf{I},X}=\Id$. Example~\ref{ex:N-graded-Lambda} supplies nontrivial $\Lambda$ satisfying (D4) in a setting where $\mathrm{Nat}(\Id,\Id)$ is large; by contrast, in categories such as $\mathbf{Set}$ or $\mathbf{Vect}_k$ one has only the trivial (scalar) endomorphisms of $\Id$, whence $\Lambda=\Id$ necessarily. This explains why “forcing” nontrivial $\Lambda$ in those categories fails, and it corrects the flawed set-based counterexample in the original draft.
\end{remark}

\section{Discussion and Outlook}
\label{sec:discussion}

\subsection{A 2-monad presentation of distorted monoidal categories}
\label{sec:2monad-presentation}

We make precise the 2\mbox{-}monadic slogan mentioned in the Introduction. 
Let $\mathsf S$ denote the 2\mbox{-}monad on $\mathbf{Cat}$ whose \emph{strict} algebras are strict monoidal categories and whose \emph{pseudo}algebras are (weak) monoidal categories \cite[§3–§4]{Lack2010}, cf.\ \cite{Benabou1967}. 
Throughout the paper we work at the level of $\mathsf S$–\emph{pseudo}algebras.

\paragraph{Adjoining the distortion operations.}
Form the algebraic extension $\mathsf T_\sigma$ of $\mathsf S$ by freely adding:
\begin{itemize}
\item a binatural family of 1\mbox{-}cells $\sigma_{X,Y}:X\otimes Y\to Y\otimes X$,
\item a natural family of 1\mbox{-}cells $\Lambda_X:X\to X$,
\end{itemize}
together with equational 2\mbox{-}cells imposing exactly the axioms \textnormal{(D1)}–\textnormal{(D4)} and the usual monoidal axioms. 
Since these are algebraic equations between pasting composites, the resulting algebraic 2\mbox{-}monad $\mathsf T_\sigma$ exists and is finitary by standard 2\mbox{-}monad machinery \cite{Lack2010}.

\begin{proposition}\label{prop:Tsigmaalgebras}
There is a biequivalence between:
\begin{enumerate}
\item distorted monoidal categories in the sense of Definition~\ref{def:distmoncat}, and
\item $\mathsf T_\sigma$–pseudoalgebras.
\end{enumerate}
Under this correspondence, $\sigma$–lax monoidal functors satisfying \textnormal{(S$\Lambda$)} and \textnormal{(S$\sigma$)} are precisely \emph{lax} $\mathsf T_\sigma$–morphisms, and monoidal natural transformations are $\mathsf T_\sigma$–transformations.
\end{proposition}

\begin{proof}[Proof sketch]
Given a distorted monoidal category $(\C,\otimes,\unit,\alpha,\lambda,\rho,\Lambda,\sigma)$, interpret the $\mathsf S$–part by the underlying monoidal structure and interpret the new operation symbols by $\Lambda$ and $\sigma$. 
Axioms \textnormal{(D1)}–\textnormal{(D4)} then witness the required algebraic equalities, so $(\C,\dots)$ is a $\mathsf T_\sigma$–pseudoalgebra.
Conversely, unpacking a $\mathsf T_\sigma$–pseudoalgebra yields exactly the data and axioms of Definition~\ref{def:distmoncat}.
A lax $\mathsf T_\sigma$–morphism is by definition a functor $F$ equipped with structure maps $\mul{F}{X}{Y}$ and $\munit{F}$ satisfying the lax associativity and unit axioms; the extra algebraic equations in $\mathsf T_\sigma$ force \textnormal{(S$\Lambda$)} and \textnormal{(S$\sigma$)}. 
Transformations are componentwise the usual monoidal natural transformations. 
\end{proof}

\begin{corollary}\label{cor:strict2cat-from-Tsigma}
Let $\mathbf{DistMon}_\sigma$ be the 2\mbox{-}category whose objects are $\mathsf T_\sigma$–pseudoalgebras, 1\mbox{-}cells are lax $\mathsf T_\sigma$–morphisms, and 2\mbox{-}cells are $\mathsf T_\sigma$–transformations. 
Then horizontal and vertical compositions are given by the usual whiskering and functor composition, and the composite laxators are
\[
\mul{GF}{X}{Y}:=G\!\big(\mul{F}{X}{Y}\big)\circ \mul{G}{FX}{FY},
\qquad
\munit{GF}:=G(\munit{F})\circ \munit{G}.
\]
In particular,
\[
\mul{H(GF)}_{X,Y}=\mul{(HG)F}_{X,Y}
\quad\text{and}\quad
\munit{H(GF)}=\munit{(HG)F},
\]
and the interchange law for 2\mbox{-}cells holds as an equality. 
Thus Theorem~\ref{thm:T1-2cat} is recovered as the concrete instance of the general $\mathsf T_\sigma$–picture.
\end{corollary}

\begin{remark}[Strict versus pseudo]
Earlier informal wording in the Introduction suggested that “strict algebras are monoidal categories.” 
Precisely: strict $\mathsf S$–algebras are \emph{strict} monoidal categories, while (weak) monoidal categories are \emph{pseudo} $\mathsf S$–algebras \cite{Lack2010}. 
Our development uses $\mathsf T_\sigma$–\emph{pseudo}algebras for objects and \emph{lax} $\mathsf T_\sigma$–morphisms for 1\mbox{-}cells; this choice matches the concrete laxator calculus used throughout the paper and explains the on-the-nose equalities for composite laxators.
\end{remark}

\subsection{Strict $\Lambda$–compatibility: consequences and recognition principles}

The axiom (S$\Lambda$) forces $F(\Lambda)=\Lambda' F$ \emph{as an equality}. This narrows the class of admissible 1\mbox{-}morphisms but yields strict 2\mbox{-}categorical composition. Two basic observations clarify when nontrivial $\Lambda$ exist.

\begin{proposition}[Triviality of monoidal endo\texorpdfstring{$\Id$}{Id} on \texorpdfstring{$\mathbf{Vect}_k$}{Vect}]\label{prop:Lambda-vect}
Let $\C=\mathbf{Vect}_k$ with the usual tensor. If $\Lambda:\Id_\C\Rightarrow\Id_\C$ is natural and satisfies (D4) and $\Lambda_{\mathbf{I}}=\Id$, then $\Lambda=\Id$. 
\end{proposition}

\begin{proof}[Proof sketch]
Naturality forces $\Lambda_V=c\cdot \Id_V$ for some $c\in k$ independent of $V$. Monoidality gives
\[
\Lambda_{V\otimes W}=\Lambda_V\otimes\Lambda_W \quad\Rightarrow\quad c\,\Id_{V\otimes W}=c^2\,\Id_{V\otimes W}.
\]
Hence $c\in\{0,1\}$. The unit condition $\Lambda_{\mathbf{I}}=\Id$ forces $c=1$.
\end{proof}

Thus in $\mathbf{Vect}_k$ there is no nontrivial $\Lambda$ compatible with (D4). By contrast, Example~\ref{ex:N-graded-Lambda} exhibits many nontrivial $\Lambda$ in graded settings.

\begin{proposition}[Existence in graded contexts]\label{prop:Lambda-graded}
Let $\C$ be a monoidal category whose objects carry a grading by a commutative monoid $(M,+,0)$ such that $(V\otimes W)_m=\bigoplus_{i+j=m}V_i\otimes W_j$ and $\mathbf{I}$ is concentrated in degree $0$. For any monoid morphism $\chi:M\to (R,\cdot,1)$ into the multiplicative monoid of a commutative ring $R$, the assignment
\[
\Lambda^{(\chi)}_V\big|_{V_m}:=\chi(m)\cdot \Id_{V_m}
\]
defines a natural $\Lambda^{(\chi)}$ with $\Lambda^{(\chi)}_{\mathbf{I}}=\Id$ and $\Lambda^{(\chi)}_{V\otimes W}=\Lambda^{(\chi)}_V\otimes \Lambda^{(\chi)}_W$. Hence (D4) holds.
\end{proposition}

Example~\ref{ex:N-graded-Lambda} is the case $M=\mathbb{N}$ and $\chi(m)=t^m$. This recognition principle explains why attempting to force a nontrivial $\Lambda$ in $\mathbf{Set}$ or ungraded $\mathbf{Vect}_k$ fails, correcting the earlier flawed set\mbox{-}based illustration.

\subsection{Oplax and pseudo variants; a double\texorpdfstring{$-$}{-}categorical envelope}

Our development used lax 1\mbox{-}cells. One can dually consider \emph{oplax} $\sigma$–monoidal functors with structure maps $F(X\otimes Y)\to FX\otimes' FY$ and the evident distortion compatibilities. There is a natural \emph{double category} $\mathbb{DistMon}_\sigma$ whose objects are distorted monoidal categories, horizontal arrows are lax $\sigma$–monoidal functors, vertical arrows are oplax ones, and squares are monoidal natural transformations (whiskered appropriately). Its horizontal and vertical bicategories recover the lax and oplax worlds, respectively (cf. \cite{DayStreet1997,Shulman2008}). Replacing (S$\Lambda$) by an \emph{isomorphism} $\epsilon:\Lambda'F\Rightarrow F\Lambda$ and adding the standard coherence axioms produces bicategories rather than a strict 2\mbox{-}category; this mirrors the passage from strict to weak monoidal categories and sits comfortably within the 2\mbox{-}monad framework \cite{Lack2010}.

\subsection{Coherence program}

A Mac\,Lane\mbox{-}style coherence theorem for distorted monoidal categories would state that every well\mbox{-}typed diagram built from $\alpha,\lambda,\rho,\sigma,\Lambda$ and the axioms (D1)–(D4) commutes. Even in the strict monoidal case ($\alpha=\lambda=\rho=\Id$) this is nontrivial because $\sigma$ need not be invertible. The braided hexagons are classically equivalent to the Yang–Baxter equation when $\beta$ is an isomorphism \cite{JoyalStreet1993}; in our non\mbox{-}invertible setting the typed hexagons furnish \emph{oriented} Yang–Baxter relations. A promising route is to view (D3) as a terminating, locally confluent rewriting system on string diagrams enriched by the idempotent data that appear in Proposition~\ref{prop:twist-by-e}. Establishing global confluence would yield normal forms and hence coherence. We expect that additional mild hypotheses (e.g.\ stability of chosen idempotents under tensor, as in (E2)) suffice to prove such a result. 

\subsection{Applications and case studies}

The constructions of Section~\ref{sec:examples} provide reusable patterns.

\begin{itemize}
\item \textbf{Idempotent twists of braidings} (Proposition~\ref{prop:twist-by-e}): starting from any braided category, a normalized, multiplicative idempotent family $e$ yields a non\mbox{-}invertible $\sigma=\beta\circ e$ satisfying (D1)–(D3). This covers irreversible rewiring in categorical quantum mechanics while retaining compositional control.
\item \textbf{Graded $\Lambda$'s} (Example~\ref{ex:N-graded-Lambda}): monoidal degree functions give families $\Lambda^{(\chi)}$ implementing time\mbox{-}or resource\mbox{-}aware scalings; these interact well with either symmetric or twisted $\sigma$.
\item \textbf{Directed models} (Example~\ref{ex:super-e}): parity\mbox{-}filtered idempotents produce lossy swaps that are still coherent. This aligns with directed homotopy motivations \cite{Grandis2009}.
\end{itemize}

These case studies indicate that distorted monoidal categories provide a disciplined middle ground between symmetric flexibility and directed irreversibility, with axioms designed to keep diagrammatic reasoning intact.

\subsection{Open problems}

We collect concrete problems suggested by the above analysis.

\begin{enumerate}
\item \textbf{Coherence}: Prove a Mac\,Lane\mbox{-}style coherence theorem for distorted monoidal categories. First target the strict case; then propagate to the weak case via transport of structure.
\item \textbf{Centers and half\mbox{-}braidings}: Characterize the lax center and lax half\mbox{-}braidings in $\mathbf{DistMon}_\sigma$ and relate them to classical centers in braided/symmetric contexts \cite{JoyalStreet1993,DayStreet1997}.
\item \textbf{2\mbox{-}monad semantics}: Give an explicit presentation of the 2\mbox{-}monad $T_\sigma$ and its Kleisli and Eilenberg–Moore objects, and identify when weak morphisms are created by the forgetful 2\mbox{-}functor to $\mathbf{Cat}$ \cite{Lack2010}.
\item \textbf{Enrichment}: Develop $\V$–enriched distorted monoidal categories for a distorted base $\V$, and determine when enrichment lifts along $\sigma$–lax functors.
\item \textbf{Homotopical algebra}: Construct model structures on categories of distorted monoidal categories in which weak equivalences reflect underlying categorical equivalences, and analyze Quillen adjunctions induced by idempotent twists.
\item \textbf{Computational tools}: Implement normal\mbox{-}form and equality checking for the typed hexagons and (S$\sigma$) in a proof assistant, using the explicit constructive proofs in Section~\ref{sec:main}.
\end{enumerate}

\subsection{Concluding remarks}

The strict 2\mbox{-}category $\mathbf{DistMon}_\sigma$ furnishes a minimal yet expressive calculus for irreversible tensorial phenomena. The typed hexagons and strict $\Lambda$–compatibility isolate the algebraic heart of directed interchange while keeping composition strictly associative and unital. The examples show that the framework is neither vacuous nor ad hoc: nontrivial $\Lambda$ appear naturally in graded settings, and non\mbox{-}invertible $\sigma$ arise from normalized idempotent twists of classical braidings. The next steps are to lift the strict theory to a robust weak setting and to establish coherence theorems that justify large\mbox{-}scale diagrammatic derivations. The present strict core already supports formalization and mechanized reasoning \cite{Lack2010,DayStreet1997}, and we expect it to serve as a stable substrate for further generalizations.

\section*{Acknowledgments}

The author thanks anonymous reviewers for helpful comments and suggestions. This work was supported in part by the IOTT Project.

\end{document}